\documentclass{article}
\usepackage[accepted]{aistats2014}

\usepackage{amsmath}
\usepackage{mlapa}
\usepackage{algorithm,algorithmic}
\usepackage{amssymb,amsthm,comment,bm}
\usepackage{nidanfloat}

\DeclareMathOperator*{\iintt}{\rm \iint}

\newcommand{\lo}{\mathrm{O}}
\newcommand{\ep}{\epsilon}
\newcommand{\de}{\delta}
\newcommand{\al}{\alpha}
\newcommand{\si}{\sigma}

\newcommand{\Ga}{\Gamma}
\newcommand{\id}{\,\mbox{\rm 1}\hspace{-0.63em}\mbox{\rm \small 1\,}\!}
\newcommand{\idx}[1]{\id\left[#1\right]}
\newcommand{\E}{\mathrm{E}}
\newcommand{\dmin}{D_{\mathrm{inf}}}

\newcommand{\n}{\nonumber}
\newcommand{\nn}{\nonumber\\}
\newcommand{\scap}{,\;}
\newcommand{\muhat}{\bar{x}}
\newcommand{\muhatn}[2]{\bar{x}_{#1}(#2)}
\newcommand{\muhatt}[2]{\bar{x}_{#1,#2}}

\newcommand{\vhat}{S}
\newcommand{\vhatn}[2]{\vhat_{#1}(#2)}
\newcommand{\vhatt}[2]{\vhat_{#1,#2}}
\newcommand{\sihat}{\hat{\si}}

\newcommand{\rd}{\mathrm{d}}

\newcommand{\normal}{\mathcal{N}}

\newcommand{\e}{\mathrm{e}}
\newcommand{\De}{\Delta}

\newcommand{\com}{\,,}
\newcommand{\per}{\,.}

\newcommand{\mut}{\tilde{\mu}}

\newcommand{\pt}[1]{P_{\muhat_1,\sihat_1^2}}
\newcommand{\pti}[3]{p_{#2}(#3|\samplet{#1}{#2})}
\newcommand{\ptic}[3]{\bar{p}_{#2}(#3|\samplet{#1}{#2})}
\newcommand{\eA}{\mathcal{A}}
\newcommand{\eB}{\mathcal{B}}
\newcommand{\wa}[1]{\frac{1}{1-\e^{-#1}}}

\newcommand{\regret}{\mathrm{Regret}}
\newcommand{\sample}{\hat{\theta}}
\newcommand{\samplet}[2]{\sample_{#1,#2}}
\newcommand{\samplen}[2]{\sample_{#1}(#2)}

\newtheorem{theorem}{Theorem}
\newtheorem{lemma}[theorem]{Lemma}
\newtheorem{corollary}[theorem]{Corollary}

\newtheorem{remark}{Remark}
\newtheorem{proposition}[theorem]{Proposition}

\newenvironment{proof2}[1]{\vspace{0.0mm}\noindent{\it Proof #1.}}{\vspace{0mm}}
\def\dqed{\relax\tag*{\qed}}

\begin{document}
\allowdisplaybreaks[2]
% If your paper is accepted and the title of your paper is very long,
% the style will print as headings an error message. Use the following
% command to supply a shorter title of your paper so that it can be
% used as headings.
%
%\runningtitle{I use this title instead because the last one was very long}

% If your paper is accepted and the number of authors is large, the
% style will print as headings an error message. Use the following
% command to supply a shorter version of the authors names so that
% they can be used as headings (for example, use only the surnames)
%
%\runningauthor{Surname 1, Surname 2, Surname 3, ...., Surname n}

\twocolumn[

\aistatstitle{
Optimality of Thompson Sampling
for Gaussian Bandits\\ Depends on Priors
%On Dependence of Optimality on Priors of Thompson Sampling for
%Gaussian Bandits
}
\aistatsauthor{Junya Honda \And Akimichi Takemura}

\aistatsaddress{
The University of Tokyo, Japan.\\%
%\small \{honda,takemura\}@stat.t.u-tokyo.ac.jp%
\texttt{\small \{honda,takemura\}@stat.t.u-tokyo.ac.jp}
} ]

\begin{abstract}
%The stochastic bandit problem is a model of a gambler choosing
%a slot machine with multiple arms which have differenct reward distributions.
In stochastic bandit problems, a Bayesian policy called Thompson sampling (TS)
has recently attracted much
attention for its excellent empirical performance.
% and its easiness of implementation
%for many models.
However, the theoretical analysis of this policy is difficult and
its asymptotic optimality is only proved
for one-parameter models.
In this paper we discuss the optimality of TS
for the model of normal distributions with unknown means and variances
as one of the most fundamental example of multiparameter models.
First we prove that the expected regret of TS with the uniform prior
achieves the theoretical bound,
which is the first result to show that
the asymptotic bound is achievable for the normal distribution model.
Next we prove that TS with Jeffreys prior and reference prior cannot achieve the theoretical bound.
Therefore
%This result implies that
the choice of priors is important for TS
and non-informative priors are sometimes risky in cases of multiparameter models.

\end{abstract}
\section{INTRODUCTION}
\label{sec:intro}
In reinforcement learning a tradeoff between exploration and 
exploitation of knowledge is considered.
The multiarmed bandit problem is
one formulation of the reinforcement learning and
is a model of a gambler
playing a slot machine with multiple arms.
A dilemma for the gambler is that he
%There is a tradeoff between exploration and
%exploitation of knowledge, that is, the gambler
cannot know whether
the expectation of an arm is high or not
without pulling it many times
but he suffers a loss if he pulls
suboptimal (i.e., not optimal) arms many times.
%This problem is a 
%The multiarmed bandit problem is a formulation of
%reinforcement learning, in which
%a tradeoff between exploration and
%exploitation of knowledge is considered.

This problem was formulated by \emcite{robbins}
%Robbins \cite{robbins} first considered this problem
%and later
and its theoretical bound was derived by \emcite{lai}
for single parametric models,
which was extended to multiparameter models by
\emcite{burnetas}.
These theoretical bounds show that
any suboptimal arm has to be pulled
at least logarithmic number of rounds and its coefficient is
determined by the distributions of suboptimal arms
and the expectation of the optimal arm.

Along with the asymptotic bound for this problem,
achievability of the bound has also been considered in many models.
\emcite{lai} proved the asymptotic optimality of a policy based on the notion of
{\it upper confidence bound} (UCB)
for Laplace distributions (which do not belong to exponential families)
and some exponential families including normal distributions with {\it known} variances.
The achievability of the bound was later extended to a subclass of one-parameter
exponential families \cite{kl_ucb}.

On the other hand in multiparameter or nonparametric models,
\emcite{burnetas} and \emcite{honda_colt} proved the achievability for finite-support distributions
and bounded-support distributions, respectively.
However, the above two models are both compact and
achievability of the bound is not known for non-compact
multiparameter models,
which include normal distributions with unknown means and variances.
%although it is one of the most basic settings of stochastic bandits.
%In particular, the achievablity of the bound is not proved
%for the model of known distributions with unknown means and unknown variances,
%although the achievability is proved for the normal distribution model with {\it known} variances
%in the seminal paper by \aunpcite{lai}.
Since the normal distribution model is one of the most basic settings of stochastic bandits,
many researches have been conducted for this model
%normal distribution model with unknown
%variances
\cite{burnetas,ucb,bayes_ucb}.
However, to the authors' knowledge, only the UCB-normal policy \cite{ucb}
assures a (non-optimal) logarithmic regret for this model\footnote{%
%no policy is assured a logarithmic regret
The theoretical analysis of UCB-normal contains
conjectures verified only numerically and the logarithmic regret
is not assured in the strict sense.%
}.

In this paper we discuss
the asymptotic optimality of Thompson sampling
(TS) \cite{thompson_original}
for this normal distribution model with unknown means and variances.
TS is a Bayesian policy which chooses an arm randomly according to
the posterior probability with which the arm is the optimal.
This policy was recently rediscovered and
is researched extensively
because of its excellent empirical performance for many models
%including the normal distribution model
\cite{thompson_empirical}.
The theoretical analysis of TS was first given for
Bernoulli model \cite{thompson_log,thompson} and
was later extended to general one-parameter exponential families
\cite{thompson_exponential}.

The asymptotic optimality of TS
under uniform prior is proved for
Bernoulli model in \emcite{thompson},
whereas it is proved for a more general model,
one-parameter exponential family,
under Jeffreys prior in \emcite{thompson_exponential}.
Therefore, TSs with
uniform prior and Jeffreys prior are asymptotically equivalent
at least for the Bernoulli model.
Nevertheless, we prove for the normal distribution model that 
TS with uniform prior achieves the asymptotic bound
whereas TS with Jeffreys prior and reference prior
cannot.
Furthermore, TS with Jeffreys prior cannot even achieve a logarithmic regret
and suffers a polynomial regret in expectation.
This result implies that TS may be more sensitive to the
choice of priors than expected
and non-informative priors
are sometimes risky (in other words, too optimistic) 
for multiparameter models.

This paper is organized as follows.
In Sect.\,\ref{sect_pre}, we formulate the bandit problem
for the normal distribution model
and introduce Thompson sampling.
% and introduce the theoretical bound.
%We introduce Thompson sampling and give some basic results
%on Bayes theories in Sect.\,\ref{sect_bayes}.
We give the main result on the optimality of TS in Sect.\,\ref{sect_main}.
The remaining sections are devoted to the proof of the main result.
In Sect.\,\ref{sect_prob}, we derive inequalities
for probabilities which appear in the normal distribution model.
We prove the optimality of TS with conservative priors
in Sect.\,\ref{sect_regret}
and prove the non-optimality of TS with optimistic priors
in Sect.\,\ref{sect_dame}.

\section{Preliminaries}\label{sect_pre}
%\section{Preliminaries and Asymptotic Bound}
We consider the $K$-armed stochastic bandit problem
in the normal distribution model.
The gambler pulls an arm $i\in\{1,\cdots,K\}$ at each round
and receives a reward independently and identically distributed
by $\normal(\mu_i,\si_i^2)$, where
$\normal(\mu,\si^2)$ denotes the normal distribution with mean $\mu$ and variance $\si^2$.
The gambler does not know the parameter $(\mu_i,\si_i^2)
\in \mathbb{R}\times (0,\infty)$.
The maximum expectation is denoted by
$\mu^*=\max_{i\in\{1,2,\cdots,K\}}\mu_i$.
Let $J(t)$ be the arm pulled at the $t$-th round
and $N_i(t)$ be the number of times that the arm $i$ is pulled
before the $t$-th round.
Then the regret of the gambler at the $T$-th round is given
for $\De_i=\mu^*-\mu_i$ by
\begin{align}
\regret(T)=\sum_{t=1}^T\De_{J(t)}=\sum_{i}\De_i N_i(T+1)\per\n
\end{align}
%$N_i(t)$ is the number of times
%that the arm $i$ is pulled through the first $t$ rounds.

Let $X_{i,n}$ be the $n$-th reward from the arm $i$.
We define
\begin{align}
\muhatt{i}{n}&= \frac1n\sum_{m=1}^n X_{i,m}\com\nn
\vhatt{i}{n}&=\sum_{m=1}^n (X_{i,m}-\muhatt{i}{n})^2=
\sum_{m=1}^n X_{i,m}^2-n\muhatt{i}{n}^2\com\n
\end{align}
that is,
$\muhatt{i}{n}$ and $\vhatt{i}{n}$ denote the sample mean and the
sum of squares from $n$ samples from the arm $i$, respectively.
We denote the sample mean and the sum of squares before the $t$-th round
by
$\muhatn{i}{t}= \muhatt{i}{N_i(t)}$ and
$\vhatn{i}{t}= \vhatt{i}{N_i(t)}$.
It is well known that
\begin{align}
\muhatt{i}{n}&\sim \normal(\mu_i, \si_i^2/n)\com
%\nn
\qquad
\frac{\vhatt{i}{n}}{\si_i^2}\sim\chi_{n-1}^2\com\label{normal_density}
\end{align}
where the chi-squared distribution $\chi_{n-1}^2$ with degree of freedom $n-1$
has the density
\begin{align}
\chi_{n-1}^2(s)=
\frac{s^{\frac{n-3}{2}}\e^{-\frac{s}{2}}}{2^{\frac{n-1}{2}}\Ga(\frac{n-1}{2})}\per\n
\end{align}

\subsection{Asymptotic Bound}
It is shown in \emcite{burnetas} that
under any policy satisfying a mild regularity condition
the expected regret satisfies
\begin{align}
\lefteqn{
\!\!\!
\liminf_{T\to\infty}\frac{\E[\regret(T)]}{\log T}
}\nn
&\ge \sum_{i: \De_i>0}\frac{\De_i}{
%\displaystyle
\inf_{(\mu,\si): \mu>\mu^*}
D(\normal(\mu_i,\si_i^2)\Vert\normal(\mu,\si^2))
}\com\label{bound_burnetas}
\end{align}
where $D(\cdot\Vert \cdot)$ is the KL divergence.
Since the KL divergence between normal distributions
is
\begin{align}
\lefteqn{
D(\normal(\mu_a,\si_b^2)\Vert \normal(\mu_a,\si_b^2))
%&=&
%\frac12\left(
%\frac{\si_1^2}{\si_2^2}+\frac{(\mu_2-\mu_1)^2}{\si_2^2}-\log\frac{\si_1^2}{\si_2^2}-1
%\right)\nn
}\nn
&=
\frac12\left(
\log \frac{\si_b^2}{\si_a^2}+\frac{\si_a^2+(\mu_b-\mu_a)^2}{\si_b^2}-1\n
\right)\com\n
\end{align}
the infimum in \eqref{bound_burnetas}
is expressed for $\mu_i<\mu^*$ as
\begin{align}
\frac12\log\left(1+\frac{(\mu^*-\mu_i)^2}{\si_i^2}\right)\per\n
\end{align}
Therefore, by letting
\begin{eqnarray}
\dmin(\De,\si^2)
=
\frac12\log\left(1+\frac{\De^2}{\si^2}\right)\com\n
\end{eqnarray}
we can rewrite the theoretical bound in \eqref{bound_burnetas} as
\begin{align}
\liminf_{T\to\infty}\frac{\E[\regret(T)]}{\log T}
&\ge \sum_{i: \De_i>0}\frac{\De_i}{
\dmin(\De_i,\si_i^2)}\per\label{bound_normal}
\end{align}

\subsection{Bayesian Theory and Thompson Sampling}\label{sect_bayes}
Thompson sampling is a policy based on the Bayesian viewpoint.
%and we first introduce some fundamental result on Bayesian theory
%for the normal distribution model
We mainly consider the prior $\pi(\mu_i,\si_i^2)\sim (\si_i^2)^{-1-\alpha}$,
or equivalently, $\pi(\mu_i,\si_i)\sim \si_i^{-1-2\al}$.
%uniform prior for $\mu_i$ and the prior $$,
Since the density of the inverse gamma distribution is
\begin{align}
\frac{\beta^{\al}}{\Ga(\al)}x^{-1-\al}\e^{-\frac{\beta}{x}}\com\n
\end{align}
the above prior for $\si_i^2$ corresponds to this distribution with parameters $(\al,\beta)=(\alpha, 0)$.
The cases $\al=-1, -1/2,\allowbreak 0,1/2$ correspond to uniform for parameter $\si_i^2$,
uniform for parameter $\si_i$, reference and Jeffreys priors, respectively
(see, e.g.~\emcite{bayes_robert} for results on Bayesian theory
given in this section).
Under this prior, the posterior distribution is
\begin{align}
\pi(\mu_i|\samplet{i}{n})
&\sim
\left(1+\frac{n(\mu_i-\muhatt{i}{n})^2}{\vhatt{i}{n}}\right)^{-\frac{n}{2}-\alpha}\com\n
%&\sim
%f_{n+2\al-1}(\mu_i)
\end{align}
where $\samplet{i}{n}=(\muhatt{i}{n},\vhatt{i}{n})$.
Since the density of $t$-distribution with degree of freedom $\nu$ is
\begin{align}
f_{\nu}(x)=\frac{\Ga(\frac{\nu+1}{2})}{\sqrt{\nu\pi}\Ga(\frac{\nu}{2})}
\left(1+\frac{x^2}{\nu}\right)^{-\frac{\nu+1}{2}}\com\label{posterior1}
\end{align}
we see that
\begin{align}
\pi\left(\sqrt{\frac{n(n+2\alpha-1)}{\vhatt{i}{n}}}(\mu_i-\bar{x})
\Bigg| \samplet{i}{n}\right)
= f_{n+2\al-1}\per\label{posterior2}
\end{align}
%given $\samplet{i}{n}=(\muhatt{i}{n},\vhatt{i}{n})$.

Thompson sampling is the policy which chooses an
arm randomly according to the probability
with which the arm is the optimal when each $\mu_i$ is distributed
independently
by the posterior $\pi(\mu_i|\samplen{i}{t})$
for $\samplen{i}{t}=(\muhatn{i}{t}, \vhatn{i}{t})$.
This policy is formulated as Algorithm \ref{alg_thompson}.
Note that we require $\max\{2,2-\lceil 2\al \rceil\}$ initial pulls
to avoid improper posteriors.
We use $n_0=\max\{2,3-\lceil 2\al \rceil\}$ for simplicity of the analysis.
\begin{algorithm}[t]
\caption{Thompson Sampling}\label{alg_thompson}
\begin{algorithmic}
\sonomama{Parameter:} $\al\in\mathbb{R}$.
%$n_0\in\mathbb{N}$,\, priors $\pi_i(\mu_i,\si_i^2)$ for arms $i=1,\cdots,K$.
\sonomama{Initialization:} 
\textnormal{Pull each arm $n_0=\max\{2,3-\lceil 2\al\rceil\}$ times}. % $n:=K$.
%\STATE $L_C,L_R:=\{1,\cdots,K\},\,L_N:=\emptyset,\,n:=K$.
%\STATE Pull each arm  once.
\sonomama{Loop:}
\mbox{}\\
%\vspace{-2mm}%
\begin{enumerate}
\item Sample $\mut_i(t)$ from the posterior
$\pi_i(\mu_i|\samplen{i}{t})$
under prior $\pi(\mu_i,\si_i^2)\sim (\si_i^2)^{-1-\al}$
for each arm $i$.
\item
Pull an arm $i$ maximizing
$\mut_i(t)$.
% where the tie-breaking rule is arbitrary.
\end{enumerate}
\end{algorithmic}
\end{algorithm}%

%follows $t$-distribution $f_{n+2\alpha-1}$ with degree of freedom $n+2\alpha-1$. 

\section{Regret of Thompson Sampling}\label{sect_main}
In this section we give the main result of this paper.
First we show that TS achieves the asymptotic bound for the
prior $(\si_i^2)^{-1-\al}$ with $\al<0$.
\begin{theorem}\label{thm_possible}
Let $\ep>0$ be arbitrary
and assume that there is a unique optimal arm.
Under Thompson sampling with $\al<0$, the expected regret is bounded as
%for some constant $C_{\ep}$ as
\begin{align}
%\lefteqn{
\E[\mathrm{Regret}(T)]
%}\nn
%&\le
%\sum_{i:\De_i>0}\frac{\De_i\log T}{\dmin(\De_i-\ep,\si_i^2+\ep)}+C_{\ep}\com
&\le
\sum_{i:\De_i>0}\frac{\De_i\log T}{\dmin(\De_i,\si_i^2)}+\lo((\log T)^{4/5})\per\n
%\label{thm_asymptotic}
\end{align}
%where $C_{\ep}$ is a constant independent of $T$
%such that $C_{\ep}=\lo(\ep^{-4})$ for fixed $\al$ and $\{(\mu_i,\si_i^2)\}_i$.
\end{theorem}
See Lemma \ref{lem_possible} for the specific representation of
the reminder term $\lo((\log T)^{4/5})$.
%$C_{\ep}$.
%By letting $\ep=\lo((\log T)^{-1/5})$ we see that
%\eqref{thm_asymptotic} is expressed as
%\begin{align}
%\E[\mathrm{Regret}(T)]
%&\le
%\sum_{i:\mu_i<\mu^*}\frac{\De_i\log T}{\dmin(\De_i,\si_i^2)}+\lo((\log T)^{4/5})\n
%\end{align}
%and
We see from this theorem that TS with $\al<0$ is asymptotically optimal in view of \eqref{bound_normal}.

%\begin{remark}{\rm
%A close inspection of the proof of Theorem \ref{optKL}
%reveals that
%\begin{align}
%\lefteqn{
%\E[\mathrm{Regret}(n)]
%}\nn
%&\le
%\sum_{i:\mu_i<\mu^*}
%\frac{(\mu^*-\mu_i)\log n}{\dmin((\mu_i,\si_i^2), \mu^*)-\lo(\de)}+\lo\left(\de^{-2}\right)
%\n
%\end{align}
%and we can express the reminder term $\so(\log n)$ in \eqref{thm_asymptotic}
%more precisely as $\lo((\log n)^{2/3})$
%by letting $\de=\lo((\log n)^{-1/3})$.
%}\end{remark}

Next we show that TS with $\al\ge 0$ cannot achieve the asymptotic bound.
To simplify the analysis we consider a two-armed setting more advantageous to the gambler
in which the full information on the arm 2 is known beforehand, that is,
the prior on the arm 2 is the unit point mass measure $\de_{\{(\mu_2,\si_2^2)\}}$
instead of $\pi(\mu_2,\si_2^2)\sim (\si_2^2)^{-1-\al}$.
%and $\mut_2(t)=\mu_2$ always hold.
\begin{theorem}\label{thm_impossible}
Assume that there are $K=2$ arms such that
$\mu_1>\mu_2$.
Then, under Thompson sampling such that
$\mut_1(t)\sim \pi(\mu_1 |\samplen{1}{t})$ with $\al\ge 0$
and $\mut_2(t)=\mu_2$,
%Then, under the prior $(\mu_1,\si_1^2)\sim (\si_1^2)^{-1-\al}$
%for the arm 1
%and $\de_{\{\mu_2,\si_2^2\}}$ for the arm 2,
there exists a constant $\xi>0$ independent of $\si_2$
such that
\begin{align}
\liminf_{T\to\infty}\frac{\E[\regret(T)]}{\log T}\ge \xi\per\label{cannot1}
\end{align}
In particular, if $\al>0$ then
there exist $\xi'>0$ and $\eta>0$ such that
\begin{align}
\liminf_{N\to\infty}\frac{\E[\regret(T)]}{T^{\eta}}\ge \xi'\per\label{cannot2}
\end{align}
\end{theorem}
Eq.\,\eqref{cannot2} means that TS with $\al>0$ suffers a polynomial regret
in expectation.
Also note that
the asymptotic bound in \eqref{bound_normal} approaches
zero for sufficiently small $\si_2$ in the above two-armed setting
since $\dmin(\De_i, \si_i^2)\to \infty$ as $\si_i\to 0$.
Nevertheless, the LHS of \eqref{cannot1} does not go to
zero as $\si_2\to 0$ because $\xi>0$ is independent of $\si_2$.
Therefore TS with $\al= 0$ also does not achieve the asymptotic bound
at least for sufficiently small $\si_2$.

Recall that Jeffreys and reference priors correspond
to $\al=1/2$ and $\al=0$, respectively.
Therefore this theorem means that
TS with these non-informative priors
does not achieve the asymptotic bound.

\begin{remark}{\rm
Probability that the sample mean satisfies $\muhat_i<\mu$ for any $\mu<\mu_i$
becomes large when $\si_i^2$ is large.
Therefore the posterior probability that
the true expectation $\mu_i$ is larger than $\mu>\muhat_i$
becomes large when the prior has heavy weight at large $\si_i^2$,
that is, $\al$ is small.
As a result, as $\al$ decreases, TS becomes a ``conservative'' policy
which chooses a seemingly suboptimal arm frequently.
Theorems \ref{thm_possible} and \ref{thm_impossible} mean
that the prior should be conservative to some extent
and non-informative priors are too optimistic.
}\end{remark}

\begin{remark}{\rm
Although TS with non-informative priors does not achieve the asymptotic bound
in the sense of expectation,
this fact does not necessarily mean that these priors are ``bad'' ones.
As we can see from a close inspection of the proof of Theorem \ref{thm_impossible},
the expected regret of TS with these priors becomes large
because
an enormously large regret arises with fairly small probability.
Therefore this policy performs well except for the case arising with this small probability,
and the authors think that TS with these priors also becomes a good policy
in the probably approximately correct (PAC) framework.
In any case, we should be aware that these non-informative priors
are ``risky'' in the sense of expectation.
}\end{remark}

\section{Inequalities for normal distributions and t-distributions}\label{sect_prob}
In this section we derive fundamental inequalities
for distributions appearing in Thompson sampling
for the normal distribution model.
We prove them in Appendix.

First we give a simple inequality to evaluate the ratio of gamma functions
which appears in the densities of normal, chi-squared and $t$-distributions.
\begin{lemma}\label{lem_gamma}
For $z\ge 1/2$
\begin{align}
\e^{-2/3}\le \frac{\Ga(z+\frac12)}{\Ga(z)}\le \e^{1/6}\sqrt{z}\per\n
\end{align}
\end{lemma}

Next we give
%evaluate
large deviation probabilities (see, e.g., \emcite{LDP}) for empirical means and variances.
\begin{lemma}\label{lem_ldp}
For any $\mu>\mu_i$
\begin{align}
\Pr[\muhatt{i}{n}\ge \mu]\le \e^{-\frac{n(\mu-\mu_i)^2}{2\si_i^2}}\label{ldp_mean}
\end{align}
and for any $\si^2>\si_i^2$
\begin{align}
\Pr[\vhatt{i}{n}\ge n\si^2]\le \e^{-nh\left(\frac{\si^2}{\si_i^2}\right)}\label{ldp_var}
\end{align}
where $h(x)=(x-1-\log x)/2\ge 0$.
\end{lemma}
%We prove this lemma by
%Cram\'er's theorem (see, e.g., \emcite{LDP})
%in Appendix.

\begin{remark}{\rm
It is well known that
Mill's ratio \cite[Chap.\,5]{kendall} gives a tighter bound 
for the tail probability of normal distributions,
% can be bounded tightly by 
%which can be obtained by application of integral by parts
%to the density of normal distributions.
and similar technique can also be applied to the tail weight
of $\chi^2$ distributions.
However, we use bounds in Lemma \ref{lem_ldp} based
on the large deviation principle because they are simpler
and convenient for our analysis.
%and we use this bound although 
}\end{remark}

Finally we evaluate the posterior distribution of the mean
for Thompson sampling.
Probability that the sample from the posterior is larger than or equal to
$\mu$, which is formally defined as
\begin{align}
\pti{i}{n}{\mu}
%&=\Pr[\mut_i(t)\ge \mu|\samplen{n}]
=\int_{\mu}^{\infty}\pi(x|\samplet{i}{n})\rd x\com\n
%\pi(\mut^*(l)\ge \mu|\samplet{i}{n})\com\n
\end{align}
is bounded as follows.
\begin{lemma}\label{lem_upper}
If $\mu>\muhatt{i}{n}$ and  $n\ge n_0$ then
\begin{align}
\pti{i}{n}{\mu}
&\ge
A_{n,\al}\left(1+\frac{n(\mu-\muhatt{i}{n})^2}{\vhatt{i}{n}}\right)^{-\frac{n-1}{2}-\al}\label{eq_lower}
\end{align}
and
\begin{align}
\pti{i}{n}{\mu}
&\le
%B_{\al}
\frac{\sqrt{\vhatt{i}{n}}}{\mu-\muhatt{i}{n}}
\left(1+\frac{n(\mu-\muhatt{i}{n})^2}{\vhatt{i}{n}}\right)^{-\frac{n}{2}-\al+1}\com\label{eq_upper}
\end{align}
where
\begin{align}
A_{n,\al}&=\frac{1}{2\e^{1/6}\sqrt{\pi(\frac{n}{2}+\al)}}\per\label{eq_an}
\end{align}
\end{lemma}

\section{Analysis for Conservative Priors}\label{sect_regret}
In this section we show that
Thompson sampling achieves the asymptotic bound if $\al<0$.
%For $x>1$, define
%\begin{align}
%g(x)=\sum_{t=0}^{\infty}x^{-t}=
%\frac{1}{1-x^{-t}}\per\n
%\end{align}
%Then $g(1+\ep)=\lo(\ep^{-1})$ and
The main result of this section is given as follows.
\begin{lemma}\label{lem_possible}
Fix any $\al<0$ and assume that $(\mu_1,\si_1^2)=(0,1)$
and the arm 1 is the unique optimal arm.
%and $\mu_1=\mu^*> \max_{i\neq 1}\mu_i$ hold.
Then, for any
% $\al<0$ and
$\ep<\min_{i:\De_i>0}\De_i/2$,
\begin{align}
\lefteqn{
\E[\mathrm{Regret}(T)]
}\nn
&\le
\sum_{i: \De_i>0}\De_i
\Bigg( \frac{\log T}{\dmin(\De_i-2\ep,\si_i^2+\ep)}+2-2\al
\nn
&\qquad\qquad+\frac{\sqrt{\si_i^2+\ep}}{\De_i-2\ep}
+\wa{\frac{\ep}{2\si_i^2}}+\wa{h(1+\frac{\ep}{\si_i^2})}
\Bigg)
\nn
&\quad+\De_{\max}\Bigg(
\wa{\frac{\ep^2}{2}}+\wa{h(2)}
+\frac{\mathrm{B}(1/2,-\al)}{(1-\e^{-\frac{\ep^2}{2}})^2}\nn
&\phantom{wwwwwwwwi}+
\frac{2\sqrt{2}}{\ep}
\frac{\left(1+\ep^2/8\right)^{1-\al}}{1-\left(1+\ep^2/8\right)^{-1/2}}
\Bigg)
\nn
&=
\sum_{i: \De_i>0}
\frac{\De_i\log T}{\dmin(\De_i-\ep,\si_i^2+\ep)}
+\lo(\ep^{-4})\com\n
\end{align}
where $\De_{\max}=\max_i \De_i$ and $\mathrm{B}(\cdot,\cdot)$ is the beta function.
\end{lemma}
\begin{corollary}\label{cor_possible}
Under the same assumption as Lemma \ref{lem_possible},
\begin{align}
\E[\regret(T)]\le\sum_{i: \De_i>0}
\frac{\De_i\log T}{\dmin(\De_i,\si_i^2)}
+\lo((\log T)^{4/5})\per\n
\end{align}
\end{corollary}
This corollary is straightforward from Lemma \ref{lem_possible}
with $\ep:=\lo((\log T)^{-1/5})$.

Note that $\dmin(\mu^*-\mu_i, \si_2^2)$ is invariant under the location and scale transformation,
that is,
\begin{align}
\dmin(\mu^*-\mu_i, \si_i^2)=\dmin\left(\frac{\mu^*-a}{b}-\frac{\mu_i-a}{b}, \frac{\si_i^2}{b^2}  \right).\n
\end{align}
Thus Theorem \ref{thm_possible} easily follows from
Corollary \ref{cor_possible}
by the
%location and scale
transformation
$((\mu_1,\si_1^2),(\mu_2,\si_2^2),\cdots,(\mu_K,\si_K^2))\mapsto
((0,1),\allowbreak
((\mu_2-\mu_1)/\si_1,\si_2^2/\si_1^2),\cdots,
((\mu_K-\mu_1)/{\si_1},\allowbreak \si_K^2/\si_1^2))$.

\begin{proof}[Proof of Lemma \ref{lem_possible}]
%Let $J(t)$ denote the arm pulled at the $t$-th round
%and 
Define events
\begin{align}
\eA(t)&=\{\mut^*(t)\ge -\ep\}\com\nn
\eB_i(t)&=\{\muhatn{i}{t}\le \mu_i+\de,\, \vhatn{i}{t}\le n(\si_i^2+\ep)\}\com\n
\end{align}
where $\mut^*(t)=\max_{i}\mut_i(t)$.
Then the regret at the round $T$
is bounded as
\begin{align}
\lefteqn{
\!\!\!
\mathrm{Regret}(T)
}\nn
&=
\sum_{t=1}^T \De_{J(t)}\nn
&\le
\De_{\max}\sum_{t=Kn_0+1}^T \idx{J(t)\neq 1\scap \eA^c(t)}\nn
&\quad+\sum_{i=2}^K \De_i
\left(n_0+\sum_{t=Kn_0+1}^T\idx{J(t)=i\scap \eA(t)}\right)\nn
&\le
\De_{\max}\sum_{t=Kn_0+1}^T \idx{J(t)\neq 1\scap \eA^c(t)}\nn
&\quad+
\sum_{i=2}^K\De_i\Bigg(
\sum_{t=Kn_0+1}^T\idx{J(t)=i\scap \eA(t)\scap \eB_i(t)}\nn
&\qquad\qquad+\sum_{t=Kn_0+1}^T\idx{J(t)=i\scap \eB_i^c(l)}+n_0\Bigg)
\com\label{bunkai}
\end{align}
where $\idx{\cdot}$ is the indicator function and
the superscript ``$c$'' denotes the complementary set.
In the following Lemmas \ref{lem_difficult}--\ref{lem_easy}
we bound the expectation of the above three terms
and the proof is completed.
\end{proof}
\begin{lemma}\label{lem_difficult}
If $\al<0$ then
\begin{align}
\lefteqn{
\E\left[\sum_{t=Kn_0+1}^T \idx{J(t)\neq 1 \cup \eA^c(t)}\right]
}\nn
&\le
\frac{1}{1-\e^{-\frac{\ep^2}{8}}}
+\frac{1}{1-\e^{-h(2)}}
+\frac{2\sqrt{2}}{\ep}
\frac{\left(1+\frac{\ep^2}{8}\right)^{1-\al}}{1-\left(1+\frac{\ep^2}{8}\right)^{-1/2}}\nn
&\quad+
\frac{\mathrm{B}(1/2,-\al)}{\left(1-\e^{-\frac{\ep^2}{2}}\right)^2}\nn
&=\lo(\ep^{-4})\per\n
\end{align}
\end{lemma}

\begin{lemma}\label{lem_main}
For any $i\neq 1$,
%For
%\begin{align}
%\xi_{i,\al,\de}=
%\frac{\sqrt{\si_i^2+\ep^2}}{\De_i-2\ep}
%\left(1+\frac{(\De_i-2\ep)^2}{\si_i^2+\ep^2}\right)^{1-\al}\n
%\end{align}
%We have
\begin{align}
\lefteqn{
\E\left[\sum_{t=Kn_0+1}^T\idx{J(t)=i\scap \eA(t)\scap \eB_i(t)}\right]
}\nn
&\le
\frac{\log N}{\dmin(\De_i-2\ep,\si_i^2+\ep)}
+2-2\al
+\frac{\sqrt{\si_i^2+\ep}}{\De_i-2\ep}\nn
%+\xi_{i,\al,\de}
&=\frac{\log N}{\dmin(\De_i-2\ep,\si_i^2+\ep)}+\lo(1)
\per\n
\end{align}
\end{lemma}
\begin{lemma}\label{lem_easy}
For any $i\neq 1$,
\begin{align}
\lefteqn{
\!\!\!\!\!
\E\left[\sum_{t=Kn_0+1}^T\idx{J(t)=i\scap \eB_i^c(t)}\right]
}\nn
&\le
\frac{1}{1-\e^{-\frac{\ep^2}{2\si_i^2}}}
+\frac{1}{1-\e^{-h\left(1+\frac{\ep}{\si_i^2}\right)}}
%\nn&
=\lo(\ep^{-2})\per\n
\end{align}
\end{lemma}
We prove Lemma \ref{lem_easy}
in Appendix
and prove Lemmas \ref{lem_difficult} and \ref{lem_main} in this section.

Whereas the second term of \eqref{bunkai} becomes the main term of
the regret,
the evaluation of the first term is the
most difficult point of the proof,
which corresponds to Lemma \ref{lem_difficult}.
In fact, 
it is reported in \emcite{burnetas} that
they were not able to prove the asymptotic optimality of
a policy for the normal distribution model
because of difficulty of the evaluation corresponding to
this term.
Also note that this is the term
which does not become a constant
in the case $\al\ge 0$ and is considered in the proof of
Theorem \ref{thm_impossible}.

In this paper we evaluate this term
by first bounding this term
for a fixed statistic $\samplet{i}{n}=(\muhatt{i}{n},\vhatt{i}{n})$ and
finally taking its expectation,
%by using the distribution of $\samplet{i}{n}$,
whereas a probability on this statistic
is first evaluated in \emcite{burnetas}.
By leaving the evaluation on the distribution of $\samplet{i}{n}$
to the latter part,
we can significantly simplify the integral
by variable transformation.

\begin{proof2}{of Lemma \ref{lem_difficult}}
First we bound the summation as
\begin{align}
\lefteqn{
\sum_{t=Kn_0+1}^T \idx{J(t)\neq 1 \scap \eA^c(t)}
}\nn
&=
\sum_{n=n_0}^{T}\sum_{t=Kn_0+1}^{T} \idx{J(t)\neq 1\scap \eA^c(t)\scap N_1(t)=n}\nn
&=
\sum_{n=n_0}^{T}\sum_{m=1}^{T}\nn
&\qquad\idx{m\le \sum_{t=Kn_0+1}^{T} \idx{J(t)\neq 1\scap \eA^c(t)\scap N_1(t)=n}}.\n
\end{align}
Note that
\begin{align}
m\le \sum_{t=Kn_0+1}^{T} \idx{J(t)\neq 1\scap \eA^c(t)\scap N_1(t)=n}\n
\end{align}
implies that $\mut_1(t)\le \mut^*(t)\le -\ep$ occurred for the first $m$ elements
of $\{t: \eA^c(t)\scap N_1(t)=n\}$.
Therefore,
\begin{align}
\lefteqn{
\!\!\!\!\!\!\!\!\!\!\!\!\!\!\!\!\!\!\!\!\!\!\!\!\!\!\!\!\!\!\!\!\!\!\!
\Pr\left[
m\le \sum_{t=Kn_0+1}^{T} \idx{J(t)\neq 1\scap \eA^c(t)\scap N_1(t)=n}
\right]
}\nn
&\le
(1-\pti{1}{n}{-\ep})^m\n
\end{align}
and we have
\begin{align}
\lefteqn{
\E\left[\sum_{t=Kn_0+1}^T \idx{J(t)\neq 1 \cup \eA^c(t)}\right]
}\nn
&\le
\E\left[
\sum_{n=n_0}^{T}\sum_{m=1}^{T}
(1-\pti{1}{n}{-\ep})^m
\right]\nn
&\le
\sum_{n=n_0}^{T}
\E\left[\frac{1-\pti{1}{n}{-\ep}}{\pti{1}{n}{-\ep}}
\right]\per\label{matizikan}
\end{align}
Since $\pti{i}{n}{-\ep}\ge 1/2$ for $-\ep\le \muhatt{i}{n}$
from the symmetry of $t$-distribution,
this expectation is partitioned into
\begin{align}
\lefteqn{
\E\left[\frac{1-\pti{1}{n}{-\ep}}{\pti{1}{n}{-\ep}}\right]
}\nn
&\le
2\E\left[\idx{-\ep\le\muhatt{1}{n}}(1-\pti{1}{n}{-\ep})\right]\nn
&\quad+
\E\left[\frac{\idx{\muhatt{1}{n}\le -\ep}}{\pti{1}{n}{-\ep}}\right]
\nn
&\le
%\nn&\quad
\Pr\left[-\ep<\muhatt{1}{n}\le -\ep/2\right]
\nn&\quad
+\Pr\left[-\ep/2<\muhatt{1}{n}\scap \vhatt{1}{n}\ge 2n\right]\nn
&\quad+2\E\big[\idx{-\ep/2<\muhatt{i}{n}\scap \vhatt{1}{n}\le 2n}
%\nn&\qquad\qquad\qquad\cdot
(1-\pti{i}{n}{-\ep})\big]\nn
&\quad +\E\left[\frac{\idx{\muhatt{1}{n}\le -\ep}}{\pti{1}{n}{-\ep}}\right]
.
\label{partition4}
\end{align}
From Lemma \ref{lem_ldp}, the first and second terms of \eqref{partition4}
are bounded as
\begin{align}
\Pr\left[-\ep<\muhatt{i}{n}\le -\ep/2\right]
&\le
\e^{-\frac{n\ep^2}{8}}\com\nn
%\exp\left(-\frac{n\ep^2}{8}\right)\nn
%\label{term2}\\
\Pr\left[-\ep/2<\muhatt{1}{n}\scap \vhatt{1}{n}\ge 2n\right]
&\le\e^{-nh(2)}\com\label{term3}
\end{align}
respectively.
Next,
% we evaluate the fourth term of \eqref{partition4}.
recall that $\samplet{1}{n}=(\muhatt{i}{n}, \vhatt{i}{n})$.
Then, from the symmetry of $t$-distribution
\begin{align}
%\lefteqn{
1-p_n(-\ep|\muhatt{i}{n}, \vhatt{i}{n})
%\pti{i}{n}{-\ep}
%}\nn
&=1-p_n(-\muhatt{1}{n}-\ep|0,\vhatt{1}{n})\nn
&=p_n(\muhatt{1}{n}+\ep|0,\vhatt{1}{n})\nn
&=p_n(2\muhatt{1}{n}+\ep|\muhatt{1}{n},\vhatt{1}{n})\n
%\nn
%&=p_n(2\muhatt{1}{n}+\ep|\samplet{1}{n})\per\n
%&=\Pr[\mut_1\le -\ep|\muhatt{1}{n},\vhatt{1}{n}]\nn
%&=\Pr[\mut_1-\muhatt{1}{n}\le -\ep-\muhatt{1}{n}|\muhatt{i}{n},\vhatt{i}{n}]\nn
%&=\Pr[\mut_1-\muhatt{1}{n}\ge -(-\ep-\muhatt{1}{n})|\muhatt{i}{n},\vhatt{i}{n}]\nn
%&=
%\pti{1}{n}{2\muhatt{1}{n}+\ep}\n
\end{align}
and the third term of \eqref{partition4} is bounded
from \eqref{eq_upper} as
%we obtain from \eqref{eq_upper} that
\begin{align}
\lefteqn{
\E\big[\idx{-\ep/2<\muhatt{i}{n}\scap \vhatt{1}{n}\le 2}
(1-\pti{i}{n}{-\ep})\big]
%\E\left[\idx{\mu^*-\ep/2<\muhatt{i}{n}}(1-\pti{i}{n}{\mu^*-\ep})\right]
}\nn
&=
\E\big[\idx{-\ep/2<\muhatt{i}{n}\scap \vhatt{1}{n}\le 2}
\pti{i}{n}{2\muhatt{i}{n}-\ep}\big]\nn
&\le
\frac{2\sqrt{2}}{\ep}
\left(1+\frac{\ep^2}{8}\right)^{-\frac{n}{2}-\al+1}\per\label{term4}
\end{align}
Finally we evaluate the fourth term of \eqref{partition4}.
From \eqref{normal_density} and \eqref{eq_lower}, we have
\begin{align}
\lefteqn{
\E\left[\frac{\idx{\muhatt{1}{n}\le -\ep}}{\pti{1}{n}{-\ep}}\right]
}\nn
&\le
\frac1{A_{n,\al}}\int_{-\infty}^{-\ep}\int_{0}^{\infty}
\left(1+\frac{n(x+\ep)^2}{s}\right)^{\frac{n-1}{2}+\al}\nn
&\qquad\cdot
\sqrt{\frac{n}{2\pi}}\e^{-\frac{nx^2}{2}}
\frac{s^{\frac{n-3}{2}}
\e^{-\frac{s}{2}}}{2^{\frac{n-1}{2}}\Ga(\frac{n-1}{2})}
\rd s\rd x\nn
&\le
\frac{\e^{-n\ep^2/2}}{A_{n,\al}}\int_{-\infty}^{-\ep}\int_{0}^{\infty}
\left(1+\frac{n(x+\ep)^2}{s}\right)^{\frac{n-1}{2}+\al}\nn
&\qquad\cdot
\sqrt{\frac{n}{2\pi}}\e^{-\frac{n(x+\ep)^2}{2}}
\frac{s^{\frac{n-3}{2}}
\e^{-\frac{s}{2}}}{2^{\frac{n-1}{2}}\Ga(\frac{n-1}{2})}
\rd s\rd x\label{main_sekibun}%\n
%\\
%&\phantom{wwwww}(\mbox{by $x^2\ge (x+\ep)^2+\ep^2$ for $x\le -\ep\le 0$})\n
\end{align}
by $x^2\ge (x+\ep)^2+\ep^2$ for $x\le -\ep\le 0$.
By letting
\begin{align}
(x,s)=\left(
-\ep-\sqrt{\frac{2zw}{n}},\,
2 z (1-w)\right)\com\n
\end{align}
we have
\begin{align}
\rd x \rd s&=
\det\left(
\begin{array}{cc}
-\sqrt{\frac{w}{2nz}}&2(1-w)\\
-\sqrt{\frac{z}{2nw}}&-2z
\end{array}
\right)\rd z \rd w\nn
&=
\sqrt{\frac{2z}{nw}}
\rd z \rd w\n
\end{align}
and \eqref{main_sekibun} is rewritten as
\begin{align}
\lefteqn{
\E\left[\frac{\idx{\muhatt{i}{n}\le -\ep}}{\pti{i}{n}{-\ep}}\right]
}\nn
&\le
\frac{\e^{-n\ep^2/2}}{A_{n,\al}}\int_{0}^{1}\int_{0}^{\infty}
\left(1+\frac{w}{1-w}\right)^{\frac{n-1}{2}+\al}\nn
&\qquad\cdot
\sqrt{\frac{n}{2\pi}}\e^{-zw}
\frac{(z(1-w))^{\frac{n-3}{2}}
\e^{-z(1-w)}}{2\Ga(\frac{n-1}{2})}
\sqrt{\frac{2z}{nw}}\rd z \rd w\nn
&=
\frac{\e^{-n\ep^2/2}}{2\sqrt{\pi}A_{n,\al}\Ga(\frac{n-1}{2})}\int_{0}^{1}
w^{-1/2}(1-w)^{-1-\al}\rd w
\nn
&\qquad\cdot
\int_{0}^{\infty}
\e^{-z}
z^{\frac{n}{2}-1}\rd z \nn
&=
\frac{\e^{-n\ep^2/2}}{2\sqrt{\pi}A_{n,\al}\Ga(\frac{n-1}{2})}
\mathrm{B}(1/2, -\al)
\Gamma(n/2)\nn
&\le
\frac{\e^{1/3}\sqrt{(n+2\al)(n-1)}}{2}
\e^{-n\ep^2/2}\mathrm{B}(1/2, -\al)
\nn
&\phantom{wwwwwwwwwwwwwwww}(\mbox{by Lemma \ref{lem_gamma} and \eqref{eq_an}})\nn
&\le
n\e^{-n\ep^2/2}\mathrm{B}(1/2, -\al)\per\;\,\quad
\mbox{(by $\e^{1/3}\le 2$ and $\al<0$)}
%\nn
%&\since{wwwwwwwwwwww}{by $\e^{2/3}\le 4$ and $\al<0$}
\label{term1}
\end{align}
By combining \eqref{matizikan}, \eqref{partition4},
%\eqref{term2}, 
\eqref{term3}, \eqref{term4} with \eqref{term1}, we obtain
\begin{align}
\lefteqn{
\E\left[\sum_{t=Kn_0+1}^T \idx{J(t)\neq 1 \scap \eA^c(t)}\right]
}\nn
&\le
\sum_{n=n_0}^T\Bigg(
\e^{-\frac{n\ep^2}{8}}+\e^{-nh(2)}+
\frac{2\sqrt{2}}{\ep}
\left(1+\frac{\ep^2}{8}\right)^{-\frac{n}{2}-\al+1}\nn
&\phantom{wwwwwwwwwwwwwww}+
n\e^{-n\ep^2/2}\mathrm{B}(1/2, -\al)\Bigg)\nn
&\le
\frac{1}{1-\e^{-\frac{\ep^2}{8}}}
+\frac{1}{1-\e^{-h(2)}}
+\frac{2\sqrt{2}}{\ep}
\frac{\left(1+\frac{\ep^2}{8}\right)^{1-\al}}{1-\left(1+\frac{\ep^2}{8}\right)^{-1/2}}\nn
&\quad+
\frac{\mathrm{B}(1/2,-\al)}{\left(1-\e^{-\frac{\ep^2}{2}}\right)^2}\nn
&=\lo(\ep^{-2})+\lo(1)+\lo(\ep^{-3})+\lo(\ep^{-4})=\lo(\ep^{-4})
\per\!\!\dqed\n
\end{align}
\end{proof2}
\begin{proof}[Proof of Lemma \ref{lem_main}]
Let $n_i>0$ be arbitrary.
Then
\begin{align}
\lefteqn{
\sum_{t=Kn_0+1}^T\idx{J(t)=i\scap \eA(t)\scap \eB_i(t)}
}\nn
&\le
\sum_{t=Kn_0+1}^T \idx{\mut_i(t)\ge -\ep\scap B_i(l)}\nn
&\le
n_i+
\sum_{t=Kn_0+1}^{T}
\idx{\mut_i(t)\ge -\ep\scap \eB_i(t),\,N_i(t)\ge n_i}\per\label{lognterm}
\end{align}
Under the condition $\{\eB_i(t)\scap N_i(t)=n\}$,
the probability of the event $\mut_i(t)\ge -\ep=\mu^*-\ep$
is bounded from Lemma \ref{lem_upper} as
\begin{align}
%\lefteqn{
\pti{i}{n}{-\ep}
%}\nn
&\le
%\frac{\Ga(\frac{n}{2}+\al)}{\sqrt{\pi n}(n+2\al-1)\Ga(\frac{n-1}{2}+\al)}\nn
%&\qquad\times
\frac{\sqrt{\si_i^2+\ep}}{\De_i-2\ep}
\left(1+\frac{(\De_i-2\ep)^2}{\si_i^2+\ep}\right)^{-\frac{n}{2}-\al+1}\nn
&=
\frac{\sqrt{\si_i^2+\ep}}{\De_i-2\ep}
\e^{-(n+2\al-2)\dmin(\De_i-2\ep, \si_i^2+\ep)}.\n
%\left(1+\frac{(\De_i-2\ep)^2}{\si_i^2+\ep}\right)^{-\frac{n}{2}}\n
\end{align}
Therefore the expectation of \eqref{lognterm} is bounded as
\begin{align}
\lefteqn{
\E\left[
\sum_{t=Kn_0+1}^T\idx{J(t)=i\scap \eA(t)\scap \eB_i(t)}
\right]
}\nn
&\le
n_i+
\sum_{t=Kn_0+1}^{T}
\Pr\left[\mut_i(t)\ge -\ep\scap \eB_i(t),\,N_i(t)\ge n_i\right]\nn
&\le
n_i+
T\frac{\sqrt{\si_i^2+\ep}}{\De_i-2\ep}
\e^{-(n_i+2\al-2)\dmin(\De_i-2\ep, \si_i^2+\ep)}\n
\end{align}
and we complete the proof by letting
$n_i=(\log T)/\allowbreak \dmin(\De_i-2\ep,\si_i^2+\ep)+2-2\al$.
\end{proof}

\section{Analysis for Optimistic Priors}\label{sect_dame}
In this section we prove Theorem \ref{thm_impossible}.
As mentioned before,
the evaluation in the proof
corresponds to Lemma \ref{lem_difficult},
in which $\al<0$ is required so that $\mathrm{B}(1/2,-\al)$ becomes finite.
We show in the following proof that this requirement is actually
necessary to achieve the asymptotic bound.

\begin{proof}[Proof of Theorem \ref{thm_impossible}]
We assume $(\mu_1,\si_1^2)=(0,1)$ without loss of generality.
Fix any $n\ge n_0$ and let $T_{1,n}\in\mathbb{N}\cup \{+\infty\}$ be the first round
at which
the number of samples from the arm 1 is $n$, that is,
we define $T_{1,n}=\min\{t:N_1(t)=n\}$.
Since $T_{1,n}=t$ implies that the arm $2$ is pulled $t-n-1$ times
through the first $t-1$ rounds, we have
\begin{align}
\lefteqn{
\E[\regret(T)]
%N_2(2T)
}\nn
&=
\De_i\sum_{t=1}^{\infty}\Pr[T_{1,n}=t]\E\left[
\sum_{m=1}^{2T}\idx{J(m)=2}
%N_2(2T)
\Bigg|T_{1,n}=t\right]\nn
&\ge\De_i\sum_{t=1}^{T}\Pr[T_{1,n}=t]
\E\left[\sum_{m=1}^{2T}\idx{J(m)=2}\Bigg|T_{1,n}=t\right]\nn
&\quad+\De_i\sum_{t=T+1}^{\infty}\Pr[T_{1,n}=t](T-n)\per\label{bunkai_impo}
\end{align}
%Therefore it suffices to show that
%$\E[T_2(2T)|T_{1,n}=t]$ is not sub-logarithmic in $N$ for
%all $m\le N$ uniformly.

Now we consider the case $t\le T$.
%If $T\ge n$ then
The conditional expectation in
\eqref{bunkai_impo}
is bounded as
\begin{align}
\lefteqn{
%\E[T_2(2T)|T_{1,n}=t]
\E\left[\sum_{m=1}^{2T} \idx{J(m)=2}\Bigg|T_{1,n}=t\right]
}\nn
%&=\E\left[\sum_{m=1}^{2T} \idx{J(m)=2}\Bigg|T_{1,n}=t\right]\nn
&\ge\E\left[\sum_{m=t}^{T+t-1} \idx{J(m)=2\scap N_1(m)=n}\Bigg|T_{1,n}=t\right]\n
\end{align}
Note that
%if $\{J(m)=2\scap N_1(m)=n\}$ then $N_1(m)=n$ and
if
$\{J(m)\neq 2\scap N_1(m)=n\}$ then $N_1(m')>n$
for any $m'>m$.
Therefore, for any $m\ge T_{1,n}$,
\begin{align}
%\lefteqn{
\{J(m)=2,\, N_1(m)=n\}
%}\nn
&\Leftrightarrow\!
\bigcup_{k=0}^{m-T_{1,n}}\{J(T_{1,n}+k)=2\}\nn
&\Leftarrow\!
\bigcup_{k=0}^{m-T_{1,n}}\{\mut_1(T_{1,n}+k)<\mu_2\}\n
\end{align}
since $\mut_2(t)=\mu_2$ always holds.
By defining $\ptic{1}{n}{\mu_2}\allowbreak =1-\pti{1}{n}{\mu_2}$
we have
\begin{align}
\lefteqn{
\E\left[\sum_{m=1}^{T+t-1} \idx{J(m)=2}\Bigg|T_{1,n}=t\right]
%\E[T_2(2T)|T_{1,n}=t]
}\nn
&\ge\E\left[\sum_{m=t}^{T+t-1} \idx{\bigcup_{k=0}^{m-t} \{\mut_1(t+k)<\mu_2\}}\Bigg|T_{1,n}=t\right]\nn
&=\E\left[\sum_{m=t}^{T+t-1}(\ptic{1}{n}{\mu_2})^{m-t+1}\right]\nn
&\ge\E\left[\sum_{m=1}^{T}(\ptic{1}{n}{\mu_2})^{m}\right]\nn
&=\E\left[
\left(1-(\ptic{1}{n}{\mu_2})^{T}\right)\frac{\ptic{1}{n}{\mu_2}}{\pti{1}{n}{\mu_2}}
\right]\nn
&\ge\frac12\E\left[
\idx{(\ptic{1}{n}{\mu_2})^T\le 1/2}
%\idx{\ptic{1}{n}{\mu_2}\le 2^{-\frac{1}{T}}}
\frac{\ptic{1}{n}{\mu_2}}{\pti{1}{n}{\mu_2}}
\right]\nn
&\ge\frac12\E\left[
\frac{
%\idx{\ptic{1}{n}{\mu_2}\le 2^{-\frac{1}{T}}}
\idx{(\ptic{1}{n}{\mu_2})^T\le 1/2}
}{\pti{1}{n}{\mu_2}}
\right]-\frac12\per\label{sita_id}\\
&\phantom{wwwwwwwwwwwwww}(\mbox{by $(1-p)/p= 1/p-1$})\n
\end{align}
Here we obtain from \eqref{eq_lower} that
\begin{align}
\lefteqn{
%\ptic{1}{n}{\mu_2}\le 2^{-\frac{1}{T}}
(\ptic{1}{n}{\mu_2})^T\le 1/2
}\nn
&\Leftrightarrow
\pti{1}{n}{\mu_2}\ge 1-2^{-\frac{1}{T}}\nn
&\Leftarrow
\left(1+\frac{n(\mu_2-\muhatt{1}{n})^2}{\vhatt{1}{n}}\right)^{-\frac{n-1}{2}-\al}
\ge
\frac{1-2^{-\frac{1}{T}}}{A_{n,\al}}\nn
&\Leftarrow
\left(1+\frac{n(\mu_2-\muhatt{1}{n})^2}{\vhatt{1}{n}}\right)^{-\frac{n-1}{2}-\al}
%\left(1+\frac{n(\mu_2-\bar{x}_1)^2}{S_1}\right)^{-\frac{n-1}{2}-\al}
\ge
\frac{\log 2}{A_{n,\al}T}\nn
&\phantom{wwwwwwwwwwwwwwwww}(\mbox{by $2^{x}\ge 1+x\log 2$})\nn
&\Leftrightarrow
\frac{n(\mu_2-\muhatt{1}{n})^2}{\vhatt{1}{n}}
%\frac{n(\mu_2-\bar{x}_1)^2}{S_1}
\le
\left(\frac{A_{n,\al}T}{\log 2}\right)^{\frac{1}{\frac{n-1}{2}+\al}}-1
%\nn&\phantom{wwwwwwwwwwwwwwwwwwwwwwwwww}
\;=:\;C_T\per\label{def_ct}
\end{align}
Therefore the expectation in \eqref{sita_id}
is bounded from \eqref{eq_upper}
as
\begin{align}
\lefteqn{
\E\left[
\frac{\idx{(\ptic{1}{n}{\mu_2})^T\le 1/2}}{\pti{1}{n}{\mu_2}}
\right]
}\nn
&\ge
\!\!\!
\iintt_{
\begin{array}{c}
{\scriptstyle x\le \mu_2,\,s\ge 0,}\\
{\scriptstyle \frac{n(\mu_2-x)^2}{s}\le C_T}
\end{array}
}
\!\!\!
\frac{\mu_2-x}{\sqrt{s}}
\left(1+\frac{n(\mu_2-x)^2}{s}\right)^{\frac{n}{2}+\al-1}\nn
&\phantom{wwwwwwwwwwwww}\cdot
\sqrt{\frac{n}{2\pi}}
\e^{-\frac{nx^2}{2}}
\frac{s^{\frac{n-3}{2}}\e^{-\frac{s}{2}}}{2^{\frac{n-1}{2}}\Ga(\frac{n-1}{2})}
\rd x \rd s\nn
&=
\frac{\sqrt{n}\e^{-\mu_2^2/2}}{\sqrt{\pi}2^{\frac{n}{2}}\Ga(\frac{n-1}{2})}
\!\!\!
\iintt_{
\begin{array}{c}
{\scriptstyle x\le\mu_2,\,s\ge 0,}\\
{\scriptstyle \frac{n(\mu_2-x)^2}{s}\le C_T}
\end{array}
}
\!\!\!
\e^{-\frac{n(\mu-x)^2}{2}+2\mu_2(\mu_2-x)}
\nn
&\phantom{wi}\cdot
(\mu_2-x)
\left(1+\frac{n(\mu_2-x)^2}{s}\right)^{\frac{n}{2}+\al-1}
s^{\frac{n}{2}-2}\e^{-\frac{s}{2}}
\rd x \rd s\per\n
\end{align}
By letting
\begin{align}
(x,s)=\left(
\mu_2-\sqrt{\frac{2zw}{n}},\,
2 z (1-w)\right)\com\n
\end{align}
we obtain in a similar way to \eqref{term1} that
\begin{align}
\lefteqn{
\E\left[
\frac{
%\idx{\ptic{1}{n}{\mu_2}\le 2^{-\frac{1}{T}}}
\idx{(\ptic{1}{n}{\mu_2})^T\le 1/2}
}{\pti{1}{n}{\mu_2}}
\right]
}\nn
&\ge
%\frac{\sqrt{n}\e^{-\mu_2^2/2}}{\sqrt{\pi}2^{\frac{n}{2}}\Ga(\frac{n-1}{2})}
%\int_{0}^{\frac{1}{1+\frac{1}{C_T}}} \int_{0}^{\infty}
%\e^{-z}\e^{2\mu_2\sqrt{\frac{2zw}{n}}}
%\nn
%&\qquad%\phantom{wi}
%\cdot
%\sqrt{\frac{2zw}{n}}\left(\frac{1}{1-w}\right)^{\frac{n}{2}+\al-1}
%(2z(1-w))^{\frac{n}{2}-2}\nn
%&\qquad\cdot \sqrt{\frac{2z}{nw}}\rd z \rd w\nn
%&=
\frac{\e^{-\mu_2^2/2}}{2\sqrt{\pi n}\Ga(\frac{n-1}{2})}
\int_{0}^{\frac{1}{1+\frac{1}{C_T}}} \int_{0}^{\infty}
\e^{-z}\e^{2\mu_2\sqrt{\frac{2zw}{n}}}
\nn
&\qquad%\phantom{wi}
\cdot z^{\frac{n}{2}-1}
(1-w)^{-1-\al}\rd z \rd w\nn
&\ge
\frac{\e^{-\mu_2^2/2}}{2\sqrt{\pi n}\Ga(\frac{n-1}{2})}
\int_{0}^{\infty}
\e^{2\mu_2\sqrt{\frac{2z}{n}}}
\e^{-z}z^{\frac{n}{2}-1}
\rd z\nn
&\phantom{wi}\cdot
\int_{0}^{\frac{1}{1+\frac{1}{C_T}}}
(1-w)^{-1-\al}\rd w\per
\qquad(\mbox{by $\mu_2<\mu_1=0$})
\n
\end{align}
Here note that
\begin{align}
\int_{0}^{\frac{1}{1+\frac{1}{C_T}}}
(1-w)^{-1-\al}\rd w
&=
\begin{cases}
\log (1+C_T),&\al=0,\\
\frac{(1+C_T)^{\al}-1}{\al},&\al>0.
\end{cases}\n
\end{align}
Then there exists a constant $B_{n,\al,\mu_2}$ such that
\begin{align}
\lefteqn{
\E\left[
\frac{
\idx{(\ptic{1}{n}{\mu_2})^T\le 1/2}
%\idx{\ptic{1}{n}{\mu_2}\le 2^{-\frac{1}{N+1}}}
}{\pti{1}{n}{\mu_2}}
\right]
}\nn
&\ge
\begin{cases}
B_{n,\al,\mu_2}\log (1+C_T),&\al=0,\\
B_{n,\al,\mu_2}((1+C_T)^{\al}-1),&\al>0.
\end{cases}\label{impo_owari}
\end{align}
Finally by putting \eqref{bunkai_impo}, \eqref{sita_id} and \eqref{impo_owari} together
we obtain for $\al=0$ that
\begin{align}
\lefteqn{
\E[\regret(2T)]
}\nn
&\ge
\De_2\min\left\{
T-n,\, (1/2)B_{n,\al,\mu_2}\log (1+C_T)-1/2
\right\}.\n
\end{align}
Eq.\,\eqref{cannot1} follows since $n\ge n_0$ is fixed and
$C_T$ defined in \eqref{def_ct}
is polynomial in $T$.
Eq.\,\eqref{cannot2} for $\al>0$ is obtained in the same way.
\end{proof}

\section{Conclusion}
We considered the stochastic multiarmed bandit problem
such that each reward follows a normal distribution
with an unknown mean and variance.
We proved that Thompson sampling
with prior $\pi(\mu_i,\si_i^2)\sim(\si_i^2)^{-1-\al}$
achieves the asymptotic bound if $\al<0$ but
cannot if $\al\ge 0$, which includes
reference prior $\al=0$ and Jeffreys prior $\al=1/2$.

A future work is to examine
whether
TS with non-informative priors is risky or not
for other multiparameter models
where TS is used without theoretical analysis
(see e.g., \emcite{thompson_empirical}).
Since the analysis of this paper heavily depends on
the specific form of normal distributions,
it is currently unknown whether the technique of this paper
can be applied to other models
and this generalization remains as an important open problem.

%\subsubsection*{Acknowledgements}

%Use unnumbered third level headings for the acknowledgements.  All
%acknowledgements go at the end of the paper.  Be sure to omit any
%identifying information in the initial double-blind submission!

%\subsubsection*{References}

\appendix
\section*{Appendix: Proof of Lemmas}
%\section*{Optimality of Thompson Sampling
%for Gaussian Bandits\\ Depends on Priors
%:\\ Supplementary Material}
We prove Lemmas \ref{lem_gamma}, \ref{lem_ldp}, \ref{lem_upper}
% \ref{lem_main}
and \ref{lem_easy} in this appendix.

First we prove Lemma \ref{lem_gamma} on the ratio of gamma functions.
\begin{proof}[Proof of Lemma \ref{lem_gamma}]
Since
\begin{align}
\sqrt{2\pi}z^{z-1/2}\e^{-z}\le \Ga(z)\le \sqrt{2\pi}\e^{1/6}z^{z-1/2}\e^{-z}\n
\end{align}
for $z\ge 1/2$ from Stirling's formula \cite[Sect.\,5.6(i)]{nist_handbook}, we have
\begin{align}
\frac{\Ga(z+\frac12)}{\Ga(z)}
&\ge
\e^{-2/3}\sqrt{z}\left(1+\frac1{2z}\right)^{z}\nn
&\ge
\e^{-2/3}\sqrt{1/2}\left(1+\frac1{2\cdot 1/2}\right)^{1/2}\nn
&=\e^{-2/3}\per\n
\end{align}
Similarly we have
\begin{align}
\frac{\Ga(z+\frac12)}{\Ga(z)}
&\le
\e^{-1/3}\sqrt{z}\left(1+\frac1{2z}\right)^{z}\nn
&\le
\e^{1/6}\sqrt{z}\com\n
\end{align}
which completes the proof.
\end{proof}

Next we prove Lemma \ref{lem_ldp}
based on Cram\'er's theorem \cite{LDP} given below.
\begin{proposition}[Cram\'er's theorem]
Let $Z_1,Z_2,\cdots$ be i.i.d.~random variables on $\mathbb{R}^d$.
Then, for $\bar{Z}=n^{-1}\sum_{m=1}^n Z_m\in \mathbb{R}^d$
and any convex set $C\in \mathbb{R}^d$,
\begin{align}
\Pr[\bar{Z}\in C]\le \exp\left(-\inf_{z\in C}\Lambda^*(z)\right)\com\n
\end{align}
where
\begin{align}
\Lambda^*(z)=\sup_{\lambda\in\mathbb{R}^d}\{\lambda\cdot z-\log\E[\e^{\lambda\cdot Z_1}]\}\per\n
\end{align}
\end{proposition}
\begin{proof}[Proof of Lemma \ref{lem_ldp}]
Eq.\,\eqref{ldp_mean} is straightforward
from Cram\'er's theorem with $Z_m:=X_{i,m}$
%$X_{i,1},X_{i,2},\cdots$
(see also e.g.~\npcite[Ex.\,2.2.23]{LDP}).

Now we show \eqref{ldp_var}.
Let $Z_m=(Z_m^{(1)},Z_m^{(2)}):=(X_{i,m},\, X_{i,m}^2)\allowbreak\in\mathbb{R}^2$.
Then it is easy to see that
the Fenchel-Legendre transform of the cumulant generating function of $Z_i$
is given by
\begin{align}
\lefteqn{
\Lambda^*(z^{(1)},z^{(2)})
}\nn
&=
\begin{cases}
h\left(\frac{z^{(2)}-(z^{(1)})^2)}{\si_i^2}\right)+\frac{(z^{(1)}-\mu_i)^2}{2\si_i^2},&z^{(2)}>(z^{(1)})^2,\\
%\frac12(r^{(2)}-1-\log(r^{(2)}-(r^{(1)})^2)),&r^{(2)}>(r^{(1)})^2,\\
+\infty,&z^{(2)}\le(z^{(1)})^2.
\end{cases}\n
\end{align}
Eq.\,\eqref{ldp_var} follows from
\begin{align}
\lefteqn{
\Pr[\vhatt{i}{n}\ge n\si^2]
}\nn
&=
\Pr[\bar{Z}^{(2)}-(\bar{Z}^{(1)})^2\ge \si^2]\nn
&\le
\exp\left(-n
\inf_{(z^{(1)},z^{(2)}):\,z^{(2)}-(z^{(1)})^2\ge \si^2}\Lambda^*(z^{(1)},z^{(2)})
\right)\nn
&\le
\exp\left(-n h\left(\frac{\si^2}{\si_i^2}\right)
\right)\com\n
\end{align}
where the first and the second inequalities follow because
$\{(z^{(1)},z^{(2)}):\,z^{(2)}-(z^{(1)})^2\ge \si^2\}$ is
a convex set and
$h(x)$ is increasing in $x\ge 1$, respectively.
\end{proof}

Next we prove Lemma \ref{lem_upper}
based on
Lemma \ref{lem_gamma}.
%\begin{lemma}\label{lem_gamma}
%For $z\ge 1/2$
%\begin{align}
%\e^{-2/3}\le \frac{\Ga(z+\frac12)}{\Ga(z)}\le \e^{1/6}\sqrt{z}\per\n
%\end{align}
%\end{lemma}

%The proof is straightforward from Stirling's formula \cite[Sect.\,5.6(i)]{nist_handbook}
%and given in Appendix.

\begin{proof2}{of Lemma \ref{lem_upper}}
Letting
\begin{align}
\tilde{A}&=\frac{\Ga(\frac{n}{2}+\al)}{\sqrt{\pi(n+2\al-1)}\Ga(\frac{n-1}{2}+\al)}\com\nn
x_0&=\sqrt{\frac{n(n+2\al-1)}{\vhatt{i}{n}}}(\mu-\muhatt{i}{n})\com\n
\end{align}
we can express $\pti{i}{n}{\mu}$
%$the target probability
from \eqref{posterior1} and \eqref{posterior2}
as
\begin{align}
%\lefteqn{
%\Pr[\mut_i\ge \mu]
\pti{i}{n}{\mu}
%}\nn
&=
\tilde{A}
\int_{x_0}^{\infty}
\left(1+\frac{x^2}{n+2\al-1}\right)^{-\frac{n}{2}-\al}
\rd x.\label{prob_target}
\end{align}
This integral is bounded from below by
\begin{align}
\lefteqn{
%\Pr[\mut_i\ge \mu]
\pti{i}{n}{\mu}
}\nn
&=
\tilde{A}\int_{x_0}^{\infty}
\left(1+\frac{x^2}{n+2\al-1}\right)^{\frac12}\nn
&\qquad\qquad\qquad\qquad\cdot
\left(1+\frac{x^2}{n+2\al-1}\right)^{-\frac{n+1}{2}-\al}
\rd x\nn
&\ge
\tilde{A}\int_{x_0}^{\infty}
\frac{x}{\sqrt{n+2\al-1}}\left(1+\frac{x^2}{n+2\al-1}\right)^{-\frac{n+1}{2}-\al}
\rd x\nn
&\phantom{wwwwwwwwwwwwwwwwwwww}(\mbox{by $\sqrt{1+u^2}\ge u$})\nn
&=
\frac{\tilde{A}}{\sqrt{n+2\al-1}}\left(1+\frac{x_0^2}{n+2\al-1}\right)^{-\frac{n-1}{2}-\al}\per\n
\end{align}
From Lemma \ref{lem_gamma}
\begin{align}
\frac{\tilde{A}}{\sqrt{n+2\al-1}}
%&=
%\frac{\Ga(\frac{n}{2}+\al)}{\sqrt{\pi}(n+2\al-1)\Ga(\frac{n-1}{2}+\al)}\nn
&=
\frac{\Ga(\frac{n}{2}+\al)}{2\sqrt{\pi}\Ga(\frac{n+1}{2}+\al)}\nn
&\ge
\frac{1}{2\e^{1/6}\sqrt{\pi(\frac{n}{2}+\al)}}\n
\end{align}
and we obtain \eqref{eq_lower}\per

On the other hand, the integral \eqref{prob_target}
is bounded from above by
%Then, from \eqref{posterior1} and \eqref{posterior2},
\begin{align}
\lefteqn{
%\!\!\Pr[\mut_i\ge \mu]
\pti{i}{n}{\mu}
}\nn
&=
\tilde{A}\int_{x_0}^{\infty}
\frac1x\cdot x\left(1+\frac{x^2}{n+2\al-1}\right)^{-\frac{n}{2}-\al}
\rd x\nn
&=
\tilde{A}\left[
\frac1x\cdot \frac{\frac{n-1}{2}+\al}{\frac{n-2}{2}+\al}\left(1+\frac{x^2}{n+2\al-1}\right)^{-\frac{n}{2}-\al+1}
\right]_{\infty}^{x_0}\nn
&\quad-\tilde{A}\int_{x_0}^{\infty}
\frac2{x^2}\frac{\frac{n-1}{2}+\al}{\frac{n-2}{2}+\al}
\left(1+\frac{x^2}{n+2\al-1}\right)^{-\frac{n}{2}-\al+1}
\rd x\nn
&\le
\frac{\tilde{A}}{x_0}\frac{\frac{n-1}{2}+\al}{\frac{n-2}{2}+\al}
\left(1+\frac{x_0^2}{n+2\al-1}\right)^{-\frac{n}{2}-\al+1}\per\n
\end{align}
From Lemma \ref{lem_gamma}
\begin{align}
\frac{\tilde{A}}{x_0}\frac{\frac{n-1}{2}+\al}{\frac{n-2}{2}+\al}
&=
\frac{\Ga(\frac{n-2}{2}+\al)}{2\sqrt{\pi n}\Ga(\frac{n-1}{2}+\al)}
\frac{\sqrt{\vhatt{i}{n}}}{\mu-\muhatt{i}{n}}\nn
&\le
\frac{1}{2\sqrt{\pi n}\e^{-2/3}}
\frac{\sqrt{\vhatt{i}{n}}}{\mu-\muhatt{i}{n}}\nn
&\le
\frac{\sqrt{\vhatt{i}{n}}}{\mu-\muhatt{i}{n}}
\n
\end{align}
and we complete the proof. \qed
\end{proof2}

Finally we prove Lemma
% \ref{lem_main} and
\ref{lem_easy}
%which are used
for the proof of Lemma \ref{lem_possible}.
%to show the optimality of TS.

\begin{proof2}{of Lemma \ref{lem_easy}}
First we have
\begin{align}
\lefteqn{
\sum_{t=Kn_0+1}^T\idx{J(t)=i\scap \eB_i^c(t)}
}\nn
&=
\sum_{n=n_0}^{T}\idx{\bigcup_{t=Kn_0+1}^T \{J(t)=i\scap \eB_i^c(t)\scap N_i(t)=n\}}\nn
&\le
\sum_{n=n_0}^{T}\idx{
\muhatt{i}{n}\ge \mu_i+\de \mbox{ or } \vhatt{i}{n}\ge n(\si_i^2+\ep)
}\per\n
\end{align}
Therefore, from Lemma \ref{lem_ldp},
\begin{align}
\lefteqn{
\!\!\!\!\!
\E\left[\sum_{t=Kn_0+1}^T\idx{J(t)=i\scap \eB_i^c(t)}\right]
}\nn
&\le
\sum_{n=n_0}^{T}\left(
\e^{-n\frac{\ep^2}{2\si_i^2}}+
\e^{-nh\left(1+\frac{\ep}{\si_i^2}\right)}
\right)\nn
&\le
\frac{1}{1-\e^{-\frac{\ep^2}{2\si_i^2}}}
+\frac{1}{1-\e^{-h\left(1+\frac{\ep}{\si_i^2}\right)}}\nn
&=\lo(\ep^{-2})+\lo(\ep^{-2})=\lo(\ep^{-2})
\per\dqed\n
\end{align}
\end{proof2}

\bibliographystyle{mlapa}

\begin{thebibliography}{}

\bibitem[Agrawal \& Goyal, 2012][Agrawal and Goyal][2012]{thompson_log}
Agrawal, S., \& Goyal, N. (2012).
\newblock Analysis of thompson sampling for the multi-armed bandit problem.
\newblock {\em Proceedings of COLT 2012}, {\em 23}, 39.1--39.26.

\bibitem[Auer et~al.\/, 2002][Auer et~al.\/][2002]{ucb}
Auer, P., Cesa-Bianchi, N., \& Fischer, P. (2002).
\newblock Finite-time analysis of the multiarmed bandit problem.
\newblock {\em Machine Learning}, {\em 47}, 235--256.

\bibitem[Burnetas \& Katehakis, 1996][Burnetas and Katehakis][1996]{burnetas}
Burnetas, A.~N., \& Katehakis, M.~N. (1996).
\newblock Optimal adaptive policies for sequential allocation problems.
\newblock {\em Advances in Applied Mathematics}, {\em 17}, 122--142.

\bibitem[Chapelle \& Li, 2012][Chapelle and Li][2012]{thompson_empirical}
Chapelle, O., \& Li, L. (2012).
\newblock An empirical evaluation of {T}hompson sampling.
\newblock {\em Proceedings of NIPS 2011} (pp.\/ 1252--1260).
\newblock Granada, Spain.

\bibitem[Dembo \& Zeitouni, 1998][Dembo and Zeitouni][1998]{LDP}
Dembo, A., \& Zeitouni, O. (1998).
\newblock {\em Large deviations techniques and applications}, vol.~38 of {\em
  Applications of Mathematics}.
\newblock New York: Springer-Verlag. Second edition.

\bibitem[Garivier \& Capp\'e, 2011][Garivier and Capp\'e][2011]{kl_ucb}
Garivier, A., \& Capp\'e, O. (2011).
\newblock The {KL-UCB} algorithm for bounded stochastic bandits and beyond.
\newblock {\em Proceedings of COLT 2011}.
\newblock Budapest, Hungary.

\bibitem[Honda \& Takemura, 2010][Honda and Takemura][2010]{honda_colt}
Honda, J., \& Takemura, A. (2010).
\newblock An asymptotically optimal bandit algorithm for bounded support
  models.
\newblock {\em Proceedings of COLT 2010} (pp.\/ 67--79).
\newblock Haifa, Israel.

\bibitem[Kaufmann et~al.\/, 2012a][Kaufmann et~al.\/][2012a]{bayes_ucb}
Kaufmann, E., Capp{\'e}, O., \& Garivier, A. (2012a).
\newblock On bayesian upper confidence bounds for bandit problems.
\newblock {\em Proceedings of Fifteenth International Conference on Artificial
  Intelligence and Statistics (AISTATS2012)} (pp.\/ 592--600).

\bibitem[Kaufmann et~al.\/, 2012b][Kaufmann et~al.\/][2012b]{thompson}
Kaufmann, E., Korda, N., \& Munos, R. (2012b).
\newblock Thompson sampling: an asymptotically optimal finite-time analysis.
\newblock {\em Proceedings of the 23rd international conference on Algorithmic
  Learning Theory (ALT'12)} (pp.\/ 199--213).
\newblock Berlin, Heidelberg: Springer-Verlag.

\bibitem[Kendall \& Stuart, 1977][Kendall and Stuart][1977]{kendall}
Kendall, M.~G., \& Stuart, A. (1977).
\newblock {\em The advanced theory of statistics}, vol. 1: Distribution theory.
\newblock C. Griffin London. 4th edition.

\bibitem[Korda et~al.\/, 2013][Korda et~al.\/][2013]{thompson_exponential}
Korda, N., Kaufmann, E., \& Munos, R. (2013).
\newblock Thompson sampling for 1-dimensional exponential family bandits.
\newblock {\em Proceedings of NIPS 2013}.
\newblock Lake Tahoe, NV, USA.

\bibitem[Lai \& Robbins, 1985][Lai and Robbins][1985]{lai}
Lai, T.~L., \& Robbins, H. (1985).
\newblock Asymptotically efficient adaptive allocation rules.
\newblock {\em Advances in Applied Mathematics}, {\em 6}, 4--22.

\bibitem[Olver et~al.\/, 2010][Olver et~al.\/][2010]{nist_handbook}
Olver, F.~W., Lozier, D.~W., Boisvert, R.~F., \& Clark, C.~W. (2010).
\newblock {\em {NIST} handbook of mathematical functions}.
\newblock New York, NY, USA: Cambridge University Press. 1st edition.

\bibitem[Robbins, 1952][Robbins][1952]{robbins}
Robbins, H. (1952).
\newblock Some aspects of the sequential design of experiments.
\newblock {\em Bulletin of the American Mathematical Society}, {\em 58},
  527--35.

\bibitem[Robert, 2001][Robert][2001]{bayes_robert}
Robert, C.~P. (2001).
\newblock {\em The {B}ayesian choice}.
\newblock New York: Springer. 2nd edition.

\bibitem[Thompson, 1933][Thompson][1933]{thompson_original}
Thompson, W.~R. (1933).
\newblock On the likelihood that one unknown probability exceeds another in
  view of the evidence of two samples.
\newblock {\em Biometrika}, {\em 25}, 285--294.

\end{thebibliography}

\end{document}